\documentclass[10pt,a4paper]{amsart}
\usepackage{amsmath,amsthm,amssymb,amsfonts}

\usepackage[pdftex]{color}
\usepackage[bookmarks=true,hyperindex,pdftex,colorlinks,citecolor=blue,
linkcolor=blue,urlcolor=blue]{hyperref}

\usepackage{todonotes}

\usepackage[english]{babel}

\usepackage{graphicx}

\theoremstyle{plain}
\newtheorem{theorem}{Theorem}[section]
\newtheorem{cor}[theorem]{Corollary}
\newtheorem{prop}[theorem]{Proposition}
\newtheorem{lemma}[theorem]{Lemma}

\theoremstyle{definition}

\newtheorem{definition}[theorem]{Definition}

\newcommand{\R}{\mathbb{R}}
\newcommand{\N}{\mathbb{N}}
\newcommand{\C}{\mathbb{C}}
\newcommand{\K}{\mathbb{K}}
\newcommand{\eps}{\varepsilon}

\DeclareMathOperator{\re}{Re}

\DeclareMathOperator{\NA}{NA}
\DeclareMathOperator{\NRA}{NRA}

\renewcommand{\leq}{\leqslant}
\renewcommand{\geq}{\geqslant}

\title{The Bishop-Phelps-Bollob\'as property and absolute sums}

\author[Choi]{Yun Sung Choi}
\address[Choi]{Department of Mathematics, POSTECH, Pohang 790-784, Republic of Korea \newline
	\href{http://orcid.org/0000-0000-0000-0000}{}}
\email{\texttt{mathchoi@postech.ac.kr}}

\author[Dantas]{Sheldon Dantas}
\address[Dantas]{Department of Mathematics, Faculty of Electrical Engineering, Czech Technical University in Prague, Technick\'a 2, 166 27, Prague 6, Czech Republic \newline
\href{http://orcid.org/0000-0001-8117-3760}{ORCID: \texttt{0000-0001-8117-3760}  }}
\email{\texttt{gildashe@fel.cvut.cz}}

\author[Jung]{Mingu Jung}
\address[Jung]{Department of Mathematics, POSTECH, Pohang 790-784, Republic of Korea \newline
	\href{http://orcid.org/0000-0003-2240-2855}{ORCID: \texttt{0000-0003-2240-2855}  }}
\email{\texttt{jmingoo@postech.ac.kr}}

\author[Mart\'{\i}n]{Miguel Mart\'{\i}n}
\address[Mart\'{\i}n]{Departamento de An\'{a}lisis Matem\'{a}tico, Facultad de
	Ciencias, Universidad de Granada, 18071 Granada, Spain \newline
	\href{http://orcid.org/0000-0003-4502-798X}{ORCID: \texttt{0000-0003-4502-798X} }
}
\email{\texttt{mmartins@ugr.es}}

\thanks{The first author was supported by Basic Science Research Program through the National Research Foundation of Korea(NRF) funded by the Ministry of Education (NRF-2015R1D1A1A09059788). The second author was supported by the project OPVVV CAAS CZ.02.1.01/0.0/0.0/16\_019/0000778, Centrum pokro\v{c}il\'ych aplikovan\'ych p\v{r}\'irodn\'ich v\v{e}d (Center for Advanced Applied Science) and by Pohang Mathematics Institute (PMI), POSTECH, Korea and Basic Science Research Program through the National Research Foundation of Korea (NRF) funded by the Ministry of Education, Science and Technology (NRF-2015R1D1A1A09059788). Last author partially supported by Spanish MINECO/FEDER grant MTM2015-65020-P}


\date{June 19th, 2018}


\begin{document}
	
\begin{abstract}
In this paper we study conditions assuring that the Bishop-Phelps-Bollob\'{a}s property (BPBp, for short) is inherited by absolute summands of the range space or of the domain space. Concretely, given a pair $(X,Y)$ of Banach spaces having the BPBp,
\begin{itemize}
\item[(a)] if $Y_1$ is an absolute summand of $Y$, then $(X,Y_1)$ has the BPBp;
\item[(b)] if $X_1$ is an absolute summand of $X$ of type $1$ or $\infty$, then $(X_1,Y)$ has the BPBp.
\end{itemize}
Besides, analogous results for the BPBp for compact operators and for the density of norm attaining operators are also given. We also show that the Bishop-Phelps-Bollob\'{a}s property for numerical radius is inherited by absolute summands of type $1$ or $\infty$. Moreover, we provide analogous results for numerical radius attaining operators and for the BPBp for numerical radius for compact operators.
\end{abstract}

\maketitle

\section{Introduction \& Preliminaries}

Let $X$ be a Banach space. We denote by $B_X$ and $S_X$ the unit ball and the unit sphere of $X$, respectively. We consider the topological dual space of $X$ and we denote it by $X^*$. We say that $x^* \in X^*$ {\it attains its norm} if there is $x_0 \in S_X$ such that $|x^*(x_0)| = \|x^*\| = \sup_{x \in S_X} |x^*(x)|$. The famous Bishop-Phelps theorem \cite{BP} says that given $\eps > 0$ and $x^* \in X^*$, there exists $x_0^* \in X^*$ such that $|x_0^*(x_0)| = \| x_0^* \|$ for some $x_0 \in S_X$ and $\|x_0^* - x^*\| < \eps$. It is natural to ask if it is true also for bounded linear operators. Given two Banach spaces $X$ and $Y$, we denote by $\mathcal{L}(X, Y)$ the set of all continuous linear operators. When $Y = X$, we denote it simply by $\mathcal{L}(X)$. We say that $T \in \mathcal{L}(X, Y)$ {\it attains its norm} when there exists $x_0 \in S_X$ such that $\|Tx_0\| = \|T\| = \sup_{x \in S_X} \|T(x)\|$. We denote by $\NA(X, Y)$ the set of all norm attaining operators from $X$ to $Y$. Then, the Bishop-Phelps theorem states that $\NA(X,\K)$ is dense in $X^*$ for every Banach space $X$ (where $\K$ denotes the base field ($=\R$ or $\C$)). Trying to extend the Bishop-Phelps theorem for bounded linear operators, J.~Lindenstrauss \cite{L} showed that there are operators which can not be approximated by norm attaining ones. Therefore, in general, there is no version of the Bishop-Phelps theorem for operators. On the other hand, if $X$ is reflexive, then $\NA(X,Y)$ is dense in $\mathcal{L}(X,Y)$ for every Banach space $Y$ (actually, this holds for Banach spaces $X$ with the Radon-Nikod\'{y}m property by a result of J.~Bourgain \cite{Bour}); if $Y$ is a closed subspace of $\ell_\infty$ containing the canonical copy of $c_0$, then $\NA(X,Y)$ is dense in $\mathcal{L}(X,Y)$ for every Banach space $X$. We refer the reader to the survey \cite{Acosta-RACSAM} for a detailed account on norm attaining operators.

In 1970, Bollob\'as \cite{Bol} proved a quantitative version of the Bishop-Phelps theorem which turned out to be very useful in numerical range theory. Nowadays, this result is known as the Bishop-Phelps-Bollob\'as theorem. It can be enunciated as follows. Let $X$ be a Banach space, let $0<\eps <2$ and suppose that $x \in B_X$ and $x^* \in B_{X^*}$ satisfy $\re x^*(x) > 1 - \frac{\eps^2}{2}$. Then, there are $y \in S_X$ and $y^* \in S_{X^*}$ such that $y^*(y) = 1$, $\|y - x\| < \eps$, and $\|y^* - x^*\| < \eps$  (see \cite{CKMM} for this slightly improved version of the original one \cite{Bol}).

This result motivated M.~Acosta, R.~Aron, D.~Garc\'ia and M.~Maestre \cite{AAGM} to introduce in 2008 the following property.

\begin{definition}[\mbox{\cite[Definition 1.1]{AAGM}}]
A pair $(X, Y)$ of Banach spaces has the {\it Bishop-Phelps-Bollob\'as property} ({\it BPBp}, for short) if given $\eps > 0$, there exists $\eta(\eps) > 0$ such that whenever $T \in \mathcal{L}(X, Y)$ with $\|T\| = 1$ and $x_0 \in S_X$ are such that
	\begin{equation*}
	\|Tx_0\| > 1 - \eta(\eps),
	\end{equation*}
	there are $S \in \mathcal{L}(X, Y)$ with $\|S\| = 1$ and $x_1 \in S_X$ such that
	\begin{equation*}
	\|Sx_1\| = 1, \ \ \ \|x_1 - x_0\| < \eps, \ \ \ \text{and} \ \ \ \|S - T\| < \eps.
	\end{equation*}
In this case, we say that the pair $(X, Y)$ has the BPBp with the function $\eps \longmapsto \eta(\eps)$.

If we restrict the operators $T$ and $S$ to be compact in the above definition, then the corresponding property is called the {\it Bishop-Phelps-Bollob\'{a}s property for compact operators} ({\it BPBp for compact operators}, for short) (see \cite{DGMM}).
\end{definition}

The aim of the authors of \cite{AAGM} was to study the conditions that $X$ and $Y$ must satisfy to get a Bishop-Phelps-Bollob\'as type theorem for bounded linear operators. They characterized when the pair $(\ell_1,Y)$ has the BPBp via a geometric property of the Banach space $Y$ which is satisfied by many Banach spaces as $C(K)$, $L_1(\mu)$, but not for all Banach spaces. They also proved that $(X, Y)$ has the BPBp when $X$ and $Y$ are finite-dimensional, or when $X$ is arbitrary and $Y$ is a closed subspace of $\ell_\infty$ containing the canonical copy of $c_0$. There is a vast literature about this topic and we invite the reader to take a look at the papers cited here and the references therein, as the already cited \cite{AAGM} and \cite{AMS, ACKLM,Cho-Choi,DGMM,DKL,DKKLM,KimLeeMartin-JMAA2015}.

In this paper we are interested in the behavior of the Bishop-Phelps-Bollob\'as property and other related properties with respect to absolute sums. Our main motivation is the similar study done in \cite{ACKLM} for $c_0$-, $\ell_1$- and $\ell_\infty$-sums, which was a very useful technique to produce some important results and examples.

Before we continue, let us give the proper terminology and notation.

An {\it absolute norm} is a norm $| \cdot |_a$ in $\R^2$ such that $|(1, 0)|_a = |(0, 1)|_a = 1$ and $|(s, t)|_a = |(|s|, |t|)|_a$ for every $s, t \in \R$. Given two Banach spaces $W$ and $Z$ and an absolute norm $| \cdot |_a$, the {\it absolute sum} of $W$ and $Z$ with respect to $| \cdot |_a$, denoted by $W \oplus_a Z$, is the Banach space $W \times Z$ endowed with the norm
\begin{equation*}
\|(w, z)\|_a = |(\|w\|, \|z\|)|_a \ \ \ (w \in W, z \in Z).
\end{equation*}
It is immediate to see that $\|(w, 0)\|_a = \|w\|$ for all $w \in W$, so $W$ is isometric to the subspace $\{(w,0)\colon w\in W\}$ of $W \oplus_a Z$. It is also easy to show that
\begin{equation} \label{ine1}
\max\{\|w\|,\|z\|\} \leq \| (w,z) \|_a \leq \|w\|+\|z\|
\end{equation}
for every $(w,z)\in W\oplus_a Z$ and every absolute norm $\| \cdot \|_a$. We will say that the Banach space $W$ is an {\it absolute summand} of the Banach space $X$, if there are another Banach space $Z$ and an absolute norm $| \cdot |_a$ in $\R^2$ such that $X = W \oplus_a Z$.  This terminology extends the well-known concepts of $L$-summand and $M$-summand (see \cite{HWW}): $W$ is a {\it $L$-summand} of $X$ if there is another Banach space $Z$ such that $X = W \oplus_1 Z$; analogously, if $X = W \oplus_{\infty} Z$ for some Banach space $Z$, then we say that $W$ is a {\it $M$-summand} of $X$. For background on absolute norms and absolute sums, we refer the reader to \cite{Bonsall1,MPR1,MPR2, MPR3, Paya, Paya1}. For a more recent reference, we suggest \cite{Hardtke}, where the author studies the stability of some geometrical properties of Banach spaces by absolute sums. Examples of absolute sums are the $\ell_p$-sums for $1 \leq p \leq \infty$ associated to the $\ell_p$-norms in $\R^2$.

In his doctoral dissertation \cite{Paya}, R.~Pay\'a proposed an intuitive classification of absolute norms defined through its behavior at the unit vector $(1, 0)$ of $\R^2$ (see also \cite[p.~38]{MPR1}). Some of our results depend on this classification, so we include it here. Let us first recall some necessary definitions. For $x \in X$, let $D(X, x)$ be the set of all $x^* \in S_{X^*}$ such that $x^*(x) = \|x\|$, which is convex and nonempty by the Hahn-Banach theorem. We say that $x \in S_X$ is a {\it vertex} of $B_X$ if $D(X, x)$ separates the points of $X$ and we say that $x$ is a {\it smooth point} of $B_X$ if $D(X, x)$ is a singleton subset of $X^*$. A vertex of the unit ball is an extreme point (see, for example, the remark after \cite[Theorem 4.6]{Bonsall0}).

\begin{definition} Let $| \cdot |_a$ be an absolute norm in $\R^2$. We say that $| \cdot |_a$ is of
	\begin{itemize}
		\item[(i)] {\it type $1$} if the vector $(1, 0)$ is an vertex of $B_{(\R^2,\|\cdot\|_a)}$;
		\item[(ii)] {\it type $2$} if the vector $(1, 0)$ is a smooth and extreme point of $B_{(\R^2,\|\cdot\|_a)}$;
		\item[(iii)] {\it type $\infty$} if the vector $(1, 0)$ is not extreme point of $B_{(\R^2,\|\cdot\|_a)}$.
	\end{itemize}
\end{definition}
The $\ell_p$-norm is of type $1$ for $p=1$, of type $\infty$ for $p=\infty$, and of type $2$ for $1<p<\infty$. In subsection \ref{subsection:absolute-sum}, at the end of this introduction, we will give an account on the results of absolute sums that we will need in this paper.

In section \ref{range} we show that if $Y_1$ is an absolute summand of $Y$ and a pair $(X,Y)$ has the BPBp (resp.\ BPBp for compact operators), then so does $(X,Y_1)$. The analogous result for the density of norm attaining operators and norm attaining compact operators also hold. For domain spaces, we show in section \ref{domain} that analogous results hold for type $1$ and type $\infty$ absolute norms: if $X_1$ is an absolute summand of type $1$ or $\infty$ of a Banach space $X$ and a pair $(X,Y)$ has the BPBp (resp.\ BPBp for compact operators), then so does $(X_1,Y)$. The corresponding results for the density of norm attaining operators and norm attaining compact operators also hold for type $1$ absolute norms.

The last section of the paper (\S \ref{numericalradius}) is devoted to the study of the Bishop-Phelps-Bollob\'{a}s property for numerical radius. Let us recall the relevant notation and terminology about this. Let $X$ be a Banach space and consider the set
\begin{equation*}
\Pi(X) := \{ (x, x^*) \in S_X \times S_{X^*}\colon x^*(x) = 1 \}.	
\end{equation*}
The {\it numerical radius} of an operator $T \in \mathcal{L}(X)$ is defined by
\begin{equation*}
v(T) := \sup \{ |x^*(Tx)|\colon (x, x^*) \in \Pi(X) \}.	
\end{equation*}
It is clear that $v(T) \leq \|T\|$ for all $T \in \mathcal{L}(X)$ and that $v(\cdot)$ is a seminorm in $\mathcal{L}(X)$. We say that $T \in \mathcal{L}(X)$ {\it attains its numerical radius} (or it is a {\it numerical radius attaining operator}) if there is $(x_0, x_0^*) \in \Pi(X)$ such that $|x_0^*(Tx_0)| = v(T)$. We denote by $\NRA(X)$ the set of all numerical radius attaining operators on $X$.
We refer the reader to the classical books \cite{Bonsall0,Bonsall1} for background on numerical radius of operators and to \cite{Acosta-RACSAM,CapMM,Paya1} and the references therein for background on the study of the density of the set of numerical radius attaining operators.

Let us give the definition of two properties related to the Bishop-Phelps-Bollob\'{a}s property. We take the definitions from \cite{KLM1} although they had appeared earlier for concrete Banach spaces (see \cite{GK}). We refer to \cite{AcoFakSole,AGR,GK, KLM1,KLM} and references therein for background.

\begin{definition} \label{defBPBp-nu}
Let $X$ be a Banach space. We say that
\begin{itemize}
\item[(a)] $X$ has the \emph{BPBp for numerical radius} (\emph{BPBp-nu}, for short) if given $\eps>0$, there is $\eta(\eps) > 0$ such that whenever $T \in \mathcal{L}(X)$ with $v(T) = 1$ and $(x, x^*) \in \Pi(X)$ satisfy
\begin{equation*}
|x^*(Tx)| > 1 - \eta(\eps),
\end{equation*}
there are $S \in \mathcal{L}(X)$ with $v(S) = 1$ and $(x_0, x_0^*) \in \Pi(X)$ such that
\begin{equation*}
|x_0^*(Sx_0)| = 1, \ \ \ \|x_0^*-x^*\| < \eps, \ \ \ \|x_0 - x\| < \eps, \ \ \ \text{and} \ \ \ \|S - T\| < \eps.
\end{equation*}
\item[(b)] $X$ has the {\it weak BPBp for numerical radius} (\emph{weak BPBp-nu}, for short) if given $\eps>0$, there is $\eta(\eps) > 0$ such that whenever $T \in \mathcal{L}(X)$ with $v(T) = 1$ and $(x, x^*) \in \Pi(X)$ satisfy
\begin{equation*}
|x^*(Tx)| > 1 - \eta(\eps),
\end{equation*}
there are $S \in \mathcal{L}(X)$ and $(x_0, x_0^*) \in \Pi(X)$ such that
\begin{equation*}
|x_0^*(Sx_0)| = v(S), \ \ \ \|x_0^*-x^*\| < \eps, \ \ \ \|x_0 - x\| < \eps, \ \ \ \text{and} \ \ \ \|S - T\| < \eps.
\end{equation*}
\end{itemize}
\end{definition}

Observe that the only difference between the BPBp-nu and the weak BPBp-nu is the normalization of the numerical radius of the operator $S$ given in the first definition. Both properties imply the density of the set of numerical radius attaining operators (see Lemma \ref{lemmaNRA}). When $v(\cdot)$ is a norm, equivalent to the operator norm (and actually in more situations, see \cite{KLM1, KLM2}), both properties are equivalent. As far as we know, it is not known whether both properties are always equivalent. Let us also say that both properties have their corresponding versions for compact operators, defined in the obvious way.

In section \ref{numericalradius} we will show that if $X$ is a Banach space with the BPBp-nu and $W$ is an absolute summand of type $1$ or $\infty$ of $X$, then $W$ has the BPBp-nu. The analogous result for the weak BPBp-nu, for the BPBp-nu for compact operators, and for the weak BPBp-nu for compact operators also hold. Furthermore, we show that if $X$ is a Banach space such that $\NRA(X)$ is dense in $\mathcal{L}(X)$ and $W$ is an absolute summand of type $1$ or $\infty$ of $X$, then $\NRA(W)$ is dense in $\mathcal{L}(W)$. The same result holds for compact operators.

Let us finally say that some of our results were previously known for the particular case of $L$-summands and/or $M$-summands, but other ones are new even in this context. We will highlight in the main part of the paper of which kind is each result. Let us also mention that for $L$-summands of the domain and for $M$-summands of the range, there is a formula for the norm of the operators involving the norms of the restrictions or projections (see the proof of \cite[Lemma 2]{PS}, for instance) which makes things easier. This is no longer true for arbitrary absolute summands.

\subsection{Some background on absolute sums}\label{subsection:absolute-sum}
Let us recall some known facts on absolute sums which will be relevant in our discussion. Let $W$, $Z$ be Banach spaces and let $\|\cdot\|_a$ be an absolute norm. There exists an isometric isomorphism between $[W \oplus_a Z]^*$ and $W^* \oplus_{a^*} Z^*$, where $| \cdot |_{a^*}$ is the dual norm associated to $| \cdot |_a$, which is also absolute. The action of a functional $(w^*, z^*) \in W^* \oplus_{a^*} Z^*$ at a point $(w, z) \in W \oplus_a Z$ is given by
\begin{equation*}
\langle (w, z), (w^*, z^*) \rangle = w^*(w) + z^*(z).
\end{equation*}

We will profusely use the following useful results which were proved in \cite{Paya}.

\begin{lemma}[\mbox{\cite[Propositions 5.3, 5.5, and 5.6]{Paya}}]\label{charac} Let $| \cdot |_a$ be an absolute norm in $\R^2$. Then,
	\begin{itemize}
		\item[(a)] $| \cdot |_a$ is of type $1$ if and only if there exists $K > 0$ such that $|x| + K |y| \leq |(x, y)|_a$ for every $x,y\in \R$.
		\item[(b)] $|\cdot |_a$ is of type $\infty$ if and only if there exists $b_0 > 0$ such that $|(1, b_0)|_a = 1$ (so, $|(1,b)|_a=1$, $\forall b\leq b_0$).
		\item[(c)] $| \cdot |_a$ is of type $1$ if and only $| \cdot |_{a^*}$ is of type $\infty$.
		\item[(d)] $| \cdot |_a$ is of type $\infty$ if and only $| \cdot |_{a^*}$ is of type $1$.
	\end{itemize}
\end{lemma}

Finally, we state the following easy result (for its proof see, for example, \cite[Lemma 2.2]{G}).

\begin{lemma} \label{state0} Let $W$ and $Z$ be Banach spaces and $\oplus_a$ be any absolute sum in $\R^2$. If $(w, z) \in S_{W \oplus_a Z}$ and $(w^*, z^*) \in S_{W^* \oplus_{a^*} Z^*}$ are such that 
\begin{equation*}
\langle (w, z), (w^*, z^*) \rangle = 1,
\end{equation*}
then
\begin{equation*}	
w^*(w) = \|w^*\|\|w\| \ \ \ \text{and} \ \ \ z^*(z) = \|z^*\| \|z\|.
\end{equation*}
\end{lemma}

\section{Results on Range Spaces} \label{range}
We start this section by showing that the BPBp passes from $(X, Y)$ to $(X, Y_1)$, when $Y_1$ is an absolute summand of $Y$. This result extends \cite[Propositions 2.3 and 2.7]{ACKLM}, where the results were done for $L$- and $M$-summands, and \cite[Theorem 2.3]{G}, where it was done for the particular case of $X=\ell_1$.

\begin{theorem} \label{range1} Let $X$, $Y$ be Banach spaces and let $Y_1$ be an absolute summand of $Y$. If the pair $(X, Y)$ has the BPBp, then so does $(X, Y_1)$.
\end{theorem}

\begin{proof} Given $\eps \in (0 , 1)$, consider $\eta (\eps) > 0$ to be the BPBp function for the pair $(X, Y)$ and let $Y_2$ be such that $Y = Y_1 \oplus_a Y_2$. Let $T_1 \in \mathcal{L}(X, Y_1)$ with $\|T_1\| = 1$ and $x_0 \in S_X$ be such that
\begin{equation*}
\|T_1x_0\| > 1 - \eta (\eps).
\end{equation*}
Define $\widetilde{T} \in \mathcal{L}(X, Y)$ by $\widetilde{T}(x) = (T_1x, 0)$ for all $x \in X$. Then $\|\widetilde{T}\| = \|T\| = 1$ and
\begin{equation*}
\|\widetilde{T} x_0\|_a = \|(T_1 x_0, 0)\|_a = \|T_1 x_0\| > 1 - \eta (\eps).
\end{equation*}
Since $\|\widetilde{T}\| = 1$, $x_0 \in S_X$ and the pair $(X, Y)$ has the BPBp with $\eta$, there are $\widetilde{S} \in \mathcal{L}(X, Y)$ with $\|\widetilde{S}\| = 1$ and $x_1 \in S_X$ such that 
\begin{equation*}
\| \widetilde{S}x_1\|_a = 1, \ \ \ \|x_1 - x_0\| < \eps \ \ \ \mbox{and} \ \ \ \|\widetilde{S} - \widetilde{T} \| < \eps.
\end{equation*}

Write $\widetilde{S} = (\widetilde{S}_1, \widetilde{S}_2)$, where $\widetilde{S}_j \in \mathcal{L}(X, Y_j)$ for $j = 1, 2$. By using (\ref{ine1}), for all $x \in B_X$, we have
\begin{equation*}
\|(\widetilde{S}_1x - T_1x, \widetilde{S}_2x)\|_{\infty} \leq \|(\widetilde{S}_1x - T_1x, \widetilde{S}_2x)\|_a \leq \|\widetilde{S} - \widetilde{T}\| < \eps.
\end{equation*}
Then, $\|\widetilde{S}_1 - T_1\| < \eps$ and $\|\widetilde{S}_2\| < \eps$. Now we consider $y^* = (y_1^*, y_2^*) \in Y_1^* \oplus_{a^*} Y_2^*$ with $\|y^*\|_{a^*} = 1$ to be such that
\begin{equation*}
1 = \|\widetilde{S}x_1\| = y^*(\widetilde{S}x_1) = y_1^* (\widetilde{S}_1x_1) + y_2^* (\widetilde{S}_2x_1).
\end{equation*}
Lemma \ref{state0} gives that $y_1^* (\widetilde{S}_1x_1) = \|y_1^*\| \|\widetilde{S}_1 x_1\|$ and $y_2^* (\widetilde{S}_2x_1) = \|y_2^*\| \|\widetilde{S}_2 x_1\|$. Since
\begin{equation*}
\|y_1^*\| \| \widetilde{S}_1x_1\| =  y_1^*(\widetilde{S}_1 x_1) = 1 -  y_2^*(\widetilde{S}_2 x_1) \geq 1 - \|\widetilde{S}_2\| > 0,
\end{equation*}
we have that $y_1^* \not= 0$ and $\|\widetilde{S}_1 x_1\| \not= 0$. Define $S \in \mathcal{L}(X, Y_1)$ by
\begin{equation}\label{eq:S1}
S_1(x) := \|y_1^*\| \widetilde{S}_1x + y_2^* (\widetilde{S}_2x) \frac{\widetilde{S}_1x_1}{\|\widetilde{S}_1x_1\|} \ \ (x  \in X).
\end{equation}
Then, for every $x \in B_X$, we have that
\begin{align*}
\|S_1x\| &\leq \|y_1^*\|\|\widetilde{S}_1x\| + \|y_2^*\|\|\widetilde{S}_2x\| \\
&= \langle ( \| \widetilde{S}_1x\|, \|\widetilde{S}_2x\|), ( \|y_1^*\|, \|y_2^*\|) \rangle \\
&\leq | ( \| \widetilde{S}_1x\|, \| \widetilde{S}_2x\|)|_a | (\|y_1^*\|, \|y_2^*\|)|_{a^*} \\
&= \| (\widetilde{S}_1x, \widetilde{S}_2x)\|_a \|(y_1^*, y_2^*)\|_{a^*} \\
&= \|\widetilde{S}x\|_a \|(y_1^*, y_2^*)\|_{a^*} \leq \|\widetilde{S}\| \|y^*\|_{a^*} = 1.
\end{align*}
So, $\|S_1\| \leq 1$. On the other hand,
\begin{equation*}
\|S_1x_1\| \geq \frac{y_1^*}{\|y_1^*\|} \left(\|y_1^*\| \widetilde{S}_1x_1 + y_2^*(\widetilde{S}_2x_1) \frac{\widetilde{S}_1x_1}{\|\widetilde{S}_1x_1\|} \right)
= y_1^*(\widetilde{S}_1x_1) +  y_2^*(\widetilde{S}_2x_1) = 1.
\end{equation*}
This shows that $\|S_1\| = \|S_1x_1\| = 1$. It remains to prove that $\|S_1 - T_1\| < \eps$. Indeed, since
\begin{equation*}
1 - \|y_1^*\| \leq 1 -  y_1^*(\widetilde{S}_1x_1) =  y_2^*(\widetilde{S}_2x_1) \leq \|\widetilde{S}_2\| < \eps,
\end{equation*}
we have for all $x \in B_X$ that
\begin{equation*}
\|S_1x - T_1x\| \leq \| \|y_1^*\| \widetilde{S}_1x - \widetilde{S}_1x\| + \|\widetilde{S}_1 - T_1\| + \|S_2\| < 3 \eps
\end{equation*}
Therefore, $\|S_1 - T_1\| < 3 \eps$. Since we already have $\|x_1 - x_0\| < \eps$, we conclude that the pair $(X, Y_1)$ has the BPBp as desired.
\end{proof}

There is a property related to the BPBp for which we may also give an analogous result. A  pair $(X, Y)$ of Banach spaces has the \emph{pointwise BPB property} (see \cite[Definition 1.2]{DKL} or \cite[Definition 1.1]{DKKLM}) if given $\eps > 0$, there is $\eta(\eps) > 0$ such that whenever $T \in \mathcal{L}(X, Y)$ with $\|T\| = 1$ and $x_0 \in S_X$ satisfy $\|Tx_0\| > 1 - \eta(\eps)$, there is $S \in \mathcal{L}(X, Y)$ such that
$$
\|S\|=\|Sx_0\| = 1 \quad \text{ and } \quad \|S - T\| < \eps.
$$
It is clear that this property is stronger than the BPBp. It was proved in \cite{DKL} that if $(X, Y)$ has the pointwise BPB property for some $Y$, then the space $X$ must be uniformly smooth. The proof of Theorem \ref{range1} can be obviously adapted to the case of the pointwise BPB property. Therefore, we can state the following result.

\begin{prop} \label{pointwiserange} Let $X$, $Y$ be Banach spaces and let $Y_1$ be an absolute summand of $Y$. If the pair $(X, Y)$ has the pointwise BPB property, then so does $(X, Y_1)$.
\end{prop}

We would like to notice also that, in Theorem \ref{range1}, if one starts with a compact operator $T_1: X \longrightarrow Y_1$ and assume that the pair $(X, Y)$ has the BPBp for compact operators, then the operator $\widetilde{T}$ defined in the proof is compact, so we can continue the proof getting a compact operator $\widetilde{S}$ and, therefore, the operator $S_1: X \longrightarrow Y_1$ defined in \eqref{eq:S1} is also compact. Thus, we have the following analogous result for this class of operators. This generalizes \cite[Lemma 2.6.ii]{DGMM}, where the result was enunciated for $L$- and $M$-summands.

\begin{prop} \label{range2} Let $X$, $Y$ be Banach spaces and let $Y_1$ be an absolute summand of $Y$. If the pair $(X, Y)$ has the BPBp for compact operators, then so does $(X, Y_1)$.
\end{prop}

Let $D$ be a bounded closed convex subset of a Banach space $X$ and let $Y$ be another Banach space. Define $$\|T\|_D := \sup \{\|Tx\|: x \in D\}$$ for every $T \in \mathcal{L}(X, Y)$. In \cite{Cho-Choi} the following version of the BPBp was defined: the pair $(X, Y)$ has the {\it BPBp on $D$} if for every $\eps > 0$, there is $\eta_D (\eps) > 0$ such that whenever $T \in \mathcal{L}(X, Y)$ with $\|T\|_D = 1$ and $x \in D$ satisfy that $\|Tx\| > 1 - \eta_D(\eps)$, there are $S \in \mathcal{L}(X, Y)$ with $\|S\|_D = 1$ and $z \in D$ such that $$\|Sz\| = 1, \qquad \|z - x\| < \eps \quad \text{ and } \quad \|S - T\| < \eps$$ (the distance $\|S - T\|$ being calculated in the usual operator norm). Note that in the proof of Theorem \ref{range1}, we can work with any bounded closed convex subset $D$ of $X$ such that $D \subset B_X$ instead of $B_X$. So, we also obtain the following result which is an extension of \cite[Propositions 4.3 and 4.5]{Cho-Choi}, where the result was done for $L$- and $M$-summands.

\begin{prop} \label{range3} Let $X$, $Y$ be Banach spaces, let $D$ be a bounded closed convex subset of $X$ such that $D \subset B_X$ and let $Y_1$ be an absolute summand of $Y$. If the pair $(X, Y)$ has the BPBp on $D$, then so does $(X, Y_1)$.
\end{prop}

We next consider analogous results for norm attaining operators.

\begin{prop} \label{NA4} Let $X$, $Y$ be Banach spaces and let $Y_1$ be an absolute summand of $Y$. If $\NA(X, Y)$ is dense in $\mathcal{L}(X, Y)$, then $\NA(X, Y_1)$ is dense in $\mathcal{L}(X, Y_1)$.	
\end{prop}

\begin{proof} Let $\eps \in (0, 1)$ and $T_1 \in \mathcal{L}(X, Y_1)$ with $\|T_1\| = 1$ be given. Let $Y_2$ be such that $Y = Y_1 \oplus_a Y_2$. Define $\widetilde{T} \in \mathcal{L}(X, Y_1 \oplus_a Y_2)$ by $\widetilde{T}(x) := (T_1x, 0)$ for all $x \in X$. Then, $\|\widetilde{T}\| = \|T\| = 1$. Since $\NA(X, Y_1 \oplus_a Y_2)$ is dense in $\mathcal{L}(X, Y_1 \oplus_a Y_2)$, there are $x_0 \in S_X$ and $\widetilde{S} \in \mathcal{L}(X, Y_1 \oplus_a Y_2)$ with $\|\widetilde{S}\| = 1$ such that $\|\widetilde{S}x_0\| = 1$ and $\|\widetilde{S} - \widetilde{T}\| < \eps$. Write $\widetilde{S} = (\widetilde{S}_1, \widetilde{S}_2)$, where $\widetilde{S}_j: X \longrightarrow Y_j$ for $j = 1, 2$. By (\ref{ine1}), we have that $\|\widetilde{S}_1 - \widetilde{T}_1\| < \eps$ and $\|\widetilde{S}_2\| < \eps$. Now, take $y^* = (y_1^*, y_2^*) \in Y_1^* \oplus_{a^*} Y_2^*$ with $\|y^*\|_{a^*} = 1$ to be such that $1 = \|\widetilde{S}x_0\| = y^*(\widetilde{S}x_0)$. Then, by Lemma \ref{state0}, $y_1^*(\widetilde{S}_1x_0) = \|\widetilde{S}_1x_0\|$ and $y_2^*(\widetilde{S}_2x_0) = \|\widetilde{S}_2x_0\|$. Since
	\begin{equation*}
	\|y_1^*\| \|\widetilde{S}_1x_0\| = y_1^*(\widetilde{S}_1x_0) = 1 - y_2^*(\widetilde{S}_2x_0) \geq 1 - \|\widetilde{S}_2\| > 1 - \eps > 0,
	\end{equation*}
	we may defined $S_1 \in \mathcal{L}(X, Y_1)$ by
	\begin{equation*}
	S_1(x) := \|y_1^*\| \widetilde{S}_1x + y_2^* (\widetilde{S}_2x) \frac{\widetilde{S}_1x_1}{\|\widetilde{S}_1x_1\|} \ \ (x  \in X).
	\end{equation*}
	This operator attains its norm at $x_0$ and it its close to $T_1$ (see the end of the proof of Theorem \ref{range1}).
\end{proof}

The above result was known for $L$-summands \cite[Proposition 2.9]{ACKLM} and for $M$-summands \cite[Lemma 2]{PS}.

Notice that the proof of Proposition \ref{NA4} also works for compact operators, providing the following result.

\begin{prop} \label{NA5} Let $X$, $Y$ be Banach spaces and let $Y_1$ be an absolute summand of $Y$. If $\NA(X, Y) \cap K(X, Y)$ is dense in $\mathcal{K}(X, Y)$, then $\NA(X, Y_1) \cap \mathcal{K}(X, Y_1)$ is dense in $\mathcal{K}(X, Y_1)$.	
\end{prop}

We finish the section with an small discussion about the validity of somehow  reciprocal results. It is shown in \cite[Proposition 2.4]{ACKLM} that if $X$, $Y_1$, $Y_2$ are Banach spaces and the pairs $(X,Y_1)$ and $(X,Y_2)$ have the BPBp, then so does the pair $(X,Y_1\oplus_\infty Y_2)$; the same result holds for the BPBp for compact operators \cite[Lemma 3.16]{DGMM} and for the density of norm attaining operators \cite[Lemma 2]{PS}. We do not know whether any of these reciprocal results is also true for arbitrary absolute sums, even for the case of $\ell_1$-sum.

\section{Results on Domain Spaces} \label{domain}

We first prove that if $X_1$ is an absolute summand of type $1$ or $\infty$, and the pair $(X, Y)$ has the BPBp, so does $(X_1, Y)$. This extends \cite[Proposition 2.6]{ACKLM}, where the result was shown for $L$- and $M$-summands.

\begin{theorem} \label{domain1} Let $X$, $Y$ be Banach spaces and let $X_1$ be an absolute summand of $X$ of type $1$ or $\infty$. If the pair $(X, Y)$ has the BPBp, then so does $(X_1, Y)$.
\end{theorem}

\begin{proof} Let $\eps \in (0, 1)$ be given and suppose that the pair $(X, Y)$ has the BPBp with some function $\eps \longmapsto \eta(\eps)$. Let $X_2$ be the Banach space such that $X = X_1 \oplus_a X_2$. Pick any $T \in \mathcal{L} (X_1, Y)$ with $\|T\| = 1$ and $x_1 \in S_{X_1}$ such that
\begin{equation*} 	
\|T x_1\| > 1 - \eta(\eps).
\end{equation*}
Define $\widetilde{T} \in \mathcal{L} (X, Y)$ by $\widetilde{T} (z_1, z_2) := Tz_1$ for $(z_1, z_2) \in X$. Then $\|\widetilde{T}\| = 1$. Since $(x_1, 0) \in S_X$,
\begin{equation*}
\|\widetilde{T} (x_1, 0)\| = \|Tx_1\| > 1 - \eta(\eps)
\end{equation*}
and the pair $(X, Y)$ has the BPBp with $\eta$, there are $\widetilde{S} \in \mathcal{L} (X, Y)$ with $\| \widetilde{S}\| = 1$ and $(x_1', x_2') \in S_X$ such that
\begin{equation*}
\| \widetilde{S} (x_1', x_2') \| = 1, \ \ \|(x_1', x_2') - (x_0, 0)\|_a < \eps \ \ \mbox{and} \ \ \| \widetilde{S} - \widetilde{T} \| < \eps.
\end{equation*}
Using (\ref{ine1}), we get that $\|x_1' - x_0\| < \eps$ and $\|x_2'\| < \eps$. Define $S \in \mathcal{L} (X_1, Y)$ by
\begin{equation*}
S(z_1) := \widetilde{S} (z_1, 0)  \ \ \ (z_1 \in X_1).
\end{equation*}
Then $\|S\| \leq \|\widetilde{S}\| = 1$ and $\|S - T\| \leq \|\widetilde{S} - \widetilde{T}\| < \eps$. To finish the proof, we will prove that $S$ attains its norm at $x_1'$ since we already have $\|x_1' - x_0\| < \eps$ and $\|S - T\| < \eps$. To do so, we divide the proof in two cases.

\noindent
{\it Case 1}: Suppose that $\oplus_a$ is an absolute norm of type $1$. By Lemma \ref{charac}.a, there exists $K > 0$ such that
\begin{equation*}
\|x_1'\| + K \| x_2'\| \leq | (\|x_1'\|, \|x_2'\|)|_a = \|(x_1', x_2')\|_a = 1.
\end{equation*}
We prove that $x_2' = 0$. Note that for all $z_2 \in B_{X_2}$, we have
\begin{equation*}
\| \widetilde{S} (0, z_2) \| = \| \widetilde{S} (0, z_2) - \widetilde{T} (0, z_2) \| \leq \| \widetilde{S} - \widetilde{T} \| < \eps.
\end{equation*}
Therefore, if we assume that $x_2' \not= 0$, we get for all $\eps \in (0, K)$ that
\begin{align*}
1 = \| \widetilde{S} (x_1', x_2') \| &= \|x_1'\| \left\| \widetilde{S} \left( \frac{x_1'}{\|x_1'\|}, 0 \right) \right\| + \|x_2'\| \left\| \widetilde{S} \left( 0, \frac{x_2'}{\|x_2'\|} \right) \right\| \\
&\leq \|x_1'\| + \eps \|x_2'\| \\
&< \|x_1'| + K \|x_2'\| \leq 1
\end{align*}

which is a contradiction. Then, $\|S\| = \| S(x_1') \| = \|\widetilde{S} (x_1', 0) \| = 1$.

\noindent
{\it Case 2}: Now assume that $\oplus_a$ is an absolute norm of type $\infty$. By Lemma \ref{charac}.b, there is $b_0 > 0$ such that $|(1, b_0)|_a = 1$. Set $\rho = \frac{b_0}{\eps} > 0$ and consider the vector $(x_1', \rho x_2') \in X$. Note that since $\|x_2'\| < \eps$, $\| \rho x_2'\| = \rho \|x_2'\| < \rho \eps = b_0$. Therefore, since $\|(x_1', \rho x_2')\|_a = | (\|x_1'\|, \| \rho x_2' \|)|_a$, $\|x_1'\| \leq 1$ and $\| \rho x_2'\| < b_0$, we have by the definition of $b_0$ that
$\|(x_1', \rho x_2')\|_a \leq |(1, b_0)|_a = 1$ and then $(x_1', \rho x_2') \in B_X$. So, writting
\begin{equation*}
(x_1', x_2') = \left( 1 - \frac{\eps}{b_0} \right) (x_1', 0) + \frac{\eps}{b_0} (x_1', \rho x_2'),
\end{equation*}
we get
\begin{equation*}
1 = \|(x_1', x_2')\|_a \leq \left( 1 - \frac{\eps}{b_0}\right) \|(x_1', 0)\|_a + \frac{\eps}{b_0} \|(x_1', \rho x_2')\|_a \leq 1
\end{equation*}
and
\begin{equation*}
1 = \|\widetilde{S} (x_1', x_2') \| \leq \left( 1 - \frac{\eps}{b_0} \right) \|\widetilde{S} (x_1', 0) \| + \frac{\eps}{b_0} \| \widetilde{S} (x_1, \rho x_2')\| \leq 1.
\end{equation*}
Then, $\|x_1'\| = \|(x_1', 0)\|_a = 1$ and $\|Sx_1'\| = \|\widetilde{S}(x_1', 0)\| = 1 = \|S\|$.
\end{proof}

We would like to notice that there is a more general result than Theorem \ref{domain1} for the pointwise BPB property (see \cite[Proposition 2.1]{DKKLM}), which says that if $X_1$ is one-complemented in $X$ and $(X, Y)$ has the pointwise BPB property, then so does $(X_1, Y)$. We do not know if it is possible to get such a general result for the BPBp.

On the other hand, we can easily adapt the proof of Theorem \ref{domain1}  in order to get an analogous result for the BPBp for compact operators. This extends \cite[Lemma 2.6.i]{DGMM}, where the result was enunciated for $L$- and $M$-summands.

\begin{prop} \label{domain2} Let $X$, $Y$ be Banach spaces and let $X_1$ be an absolute summand of $X$ of type $1$ or $\infty$. If the pair $(X, Y)$ has the BPBp for compact operators, then so does $(X_1, Y)$.
\end{prop}

Let us present now the version of Theorem \ref{domain1} for norm attaining operators, but in this case we may only deal with type $1$ absolute norms. The result was previously known for $L$-summands (see \cite[Lemma 2]{PS}).

\begin{prop} \label{NA1} Let $X$, $Y$ be Banach spaces and let $X_1$ be an absolute summand of $X$ of type $1$. If $\NA(X, Y)$ is dense in $\mathcal{L}(X, Y)$, then $\NA(X_1, Y)$ is dense in $\mathcal{L}(X_1, Y)$.
\end{prop}

\begin{proof} Let $\eps \in (0, 1)$ and $T \in \mathcal{L}(X_1, Y)$ with $\|T\| = 1$ be given. Consider $X_2$ to be a Banach space such that $X = X_1 \oplus_a X_2$. Define $\widetilde{T} \in \mathcal{L}(X, Y)$ by $\widetilde{T}(z_1, z_2) := (Tz_1, 0)$ for all $(z_1, z_2) \in X_1 \oplus_a X_2$. Then $\|\widetilde{T}\| = \|T\| = 1$. Since $\NA(X, Y)$ is dense in $\mathcal{L}(X, Y)$, there are $(\widetilde{x}_1, \widetilde{x}_2) \in S_{X_1 \oplus_a X_2}$ and $\widetilde{S} \in \mathcal{L}(X, Y)$ with $\|\widetilde{S}\| = 1$ such that $\|\widetilde{S}(\widetilde{x}_1, \widetilde{x}_2)\| = 1$ and $\|\widetilde{S} - \widetilde{T}\| < \eps$. Define $S \in \mathcal{L}(X_1, Y)$ by $S(z_1) := \widetilde{S}(z_1, 0)$ for all $z_1 \in X_1$. Then, $\|S\| \leq 1$ and for all $z_1 \in S_{X_1}$, we have
	\begin{equation*}
	\|Sz_1 - Tz_1\| = \|\widetilde{S}(z_1, 0) - \widetilde{T}(z_1, 0)\| \leq \|\widetilde{S} - \widetilde{T}\| < \eps.
	\end{equation*}
	So, $\|S - T\| < \eps$. It remains to prove that $S$ attains its norm. Indeed, first notice that for all $z_2 \in B_{X_2}$, we have
	\begin{equation*}
	\|\widetilde{S}(0, z_2)\| = \| \widetilde{S}(0, z_2) - \widetilde{T}(0, z_2)\| \leq \|\widetilde{S} - \widetilde{T}\| < \eps.
	\end{equation*}	
	This implies that $\widetilde{x}_1 \not= 0$, otherwise, we would have $1 = \|\widetilde{S}(\widetilde{x}_1, \widetilde{x}_2)\| = \|\widetilde{S}(0, \widetilde{x}_2)\| < \eps$, which is a contradiction.	Since $\oplus_a$ is of type $1$, by Lemma \ref{charac}.a, there is $K > 0$ such that $\|\widetilde{x}_1\| + K \|\widetilde{x}_2\| \leq |(\|\widetilde{x}_1\|, \|\widetilde{x}_2\|)|_a = \| (\widetilde{x}_1, \widetilde{x}_2)\|_a = 1$. If $\widetilde{x}_2 \not= 0$, we have that, for every $\eps \in (0, K)$,
	\begin{equation*}
	1 = \| \widetilde{S}(\widetilde{x}_1, \widetilde{x}_2)\| \leq \|\widetilde{x}_1\| + \|\widetilde{x}_2\| \left\| \widetilde{S} \left( 0, \frac{\widetilde{x}_2}{\|\widetilde{x}_2\|} \right)\right\| < \|\widetilde{x}_1\| + \eps \| \widetilde{x}_2 \| < \|\widetilde{x}_1\| + K \|\widetilde{x}_2\| \leq 1,
	\end{equation*}
	which is a new contradiction. So, $\widetilde{x}_2 = 0$ and, then, $\|S\widetilde{x}_1\| = \| \widetilde{S}(\widetilde{x}_1, 0)\| = 1 = \|\widetilde{x}_1\|$. Therefore, $\NA(X_1, Y)$ is dense in $\mathcal{L}(X_1, Y)$.	
\end{proof}

With the same proof, when one restricts it to compact operators, we get the following result.

\begin{prop} \label{KNA-domain} Let $X$, $Y$ be Banach spaces and let $X_1$ be an absolute summand of type $1$ of $X$. If the set $\NA(X, Y) \cap K(X, Y)$ is dense in $\mathcal{K}(X, Y)$, then $\NA(X_1, Y) \cap \mathcal{K}(X_1, Y)$ is dense in $\mathcal{K}(X_1, Y)$.	
\end{prop}

We do not know if the analogous result of Proposition \ref{NA1} holds true also for absolute norms of type $\infty$. Actually, we do not know what happens even for $M$-summands.

As in the previous case, we finish the section with an small discussion about the validity of reciprocal results. Let $X_1$, $X_2$, $Y$ be Banach spaces. It is shown in \cite[Lemma 2]{PS} that if $\NA(X_1,Y)$ and $\NA(X_2,Y)$ are dense in their respective spaces of operators, then $\NA(X_1\oplus_1 X_2,Y)$ is dense in $\mathcal{L}(X_1\oplus_1 X_2,Y)$. The validity of the analogous result for the BPBp is not true: the pair $(\R,Y)$ has the BPBp for every Banach space $Y$ (trivial), while there are $Y$'s such that $(\R\oplus_1 \R,Y)$ does not have the BPBp (see \cite[Corollary 3.3]{ACKLM} for instance). As $\R\oplus_1\R \equiv \R\oplus_\infty \R$, the same example shows that the reciprocal result is not true for the $\ell_\infty$-sum. We do not know what is the situation for $\ell_p$-sums with $1<p<\infty$.

\section{Results for Numerical Radius} \label{numericalradius}

We would like now to tackle the analogous questions of the previous sections for numerical radius. Our first result in this line is the following one which extends \cite[Lemma 19]{KLM1}, where the result was proved for $L$- and $M$-summands.

\begin{theorem} \label{numericalradius1} Let $X$ be a Banach space and let $W$ be an absolute summand of type 1 or $\infty$ of $X$. If $X$ has the BPBp-nu, so does $W$.	
\end{theorem}

We will profusely use in this section the following result which is a particular case of  \cite[Lemma 3.3]{CMM}.

\begin{lemma} \label{radius} Let $W$, $Z$ be Banach spaces and let $|\cdot|_a$ be an absolute norm. Given $T \in \mathcal{L}(W)$, we define $\widetilde{T} \in \mathcal{L}(W \oplus_a Z)$ by $\widetilde{T}(w, z) := (Tw, 0)$ for every $(w, z) \in W \oplus_a Z$. Then $v(\widetilde{T}) = v(T)$ and $\|\widetilde{T}\|=\|T\|$.
\end{lemma}

We are now able to provide the proof of Theorem \ref{numericalradius1}.

\begin{proof}[Proof of Theorem \ref{numericalradius1}] Let $\eps \in (0, 1)$ be given and suppose that $X$ has the BPBp-nu with some function $\eta(\eps) > 0$. Consider $Z$ to be a Banach space with $X = W \oplus_a Z$. We will prove that $W$ satisfies the BPBp-nu with $\eta$. Let $T \in \mathcal{L}(W)$ with $v(T) = 1$ and $(w_0, w_0^*) \in \Pi(W)$ be such that
	\begin{equation*}
	|w_0^*(Tw_0)| > 1 - \eta(\eps).
	\end{equation*}
Consider $\widetilde{T} \in \mathcal{L}(W \oplus_a Z)$ to be defined by $\widetilde{T}(w, z) = (Tw, 0)$ for every $(w, z) \in W \oplus_a Z$. Lemma \ref{radius} gives $v(\widetilde{T}) = v(T) = 1$. Also,
\begin{equation*}
|\langle \widetilde{T}(w_0, 0), (w_0^*, 0) \rangle| = |w_0^*(Tw_0)| > 1 - \eta(\eps).
\end{equation*}
Since $((w_0, 0), (w_0^*, 0)) \in \Pi(W \oplus_a Z)$, $v(\widetilde{T}) = 1$ and $W \oplus_a Z$ has the BPBp-nu with $\eta$, there are $\widetilde{S} \in \mathcal{L}(W \oplus_a Z)$ with $v(\widetilde{S}) = 1$ and $((w_1, z_1), (w_1^*, z_1^*)) \in \Pi(W \oplus_a Z)$ such that	
\begin{itemize}
\item[(a)] $|\langle \widetilde{S}(w_1, z_1), (w_1^*, z_1^*) \rangle | = 1$,
\item[(b)] $\|(w_1^*, z_1^*) - (w_0^*, 0)\|_{a^*} < \eps$,
\item[(c)] $\|(w_1, z_1) - (w_0, 0)\|_a < \eps$, and
\item[(d)] $\|\widetilde{S} - \widetilde{T}\| < \eps$.
\end{itemize}
\noindent
Write $\widetilde{S} = (\widetilde{S}_1, \widetilde{S}_2)$, where $\widetilde{S}_1: W \oplus_a Z \longrightarrow W$ and $\widetilde{S}_2: W \oplus_a Z \longrightarrow Z$. Define $S \in \mathcal{L}(W)$ by
\begin{equation*}
S(w) := \widetilde{S}_1(w, 0) \ \ \ \  \ (w \in W).
\end{equation*}
Then, for every $(w, w^*) \in \Pi(W)$, we have
\begin{equation} \label{eq1}
|w^*(Sw)| = |w^*(\widetilde{S}_1(w, 0))| = | \langle \widetilde{S}(w, 0), (w^*, 0) \rangle| \leq v(\widetilde{S}) = 1.
\end{equation}
So, $v(S) \leq 1$. Now, since $\|(w_1^*, z_1^*) - (w_0^*, 0)\|_{a^*} < \eps$ and $\|(w_1, z_1) - (w_0, 0)\|_a < \eps$, by using (\ref{ine1}), we get that
\begin{equation*}
\|w_1^* - w_0^*\| < \eps \ \ \ \mbox{and} \ \ \ \|w_1 - w_0\| < \eps.	
\end{equation*}
Moreover, for every $w \in S_W$, we have
\begin{align*}
\|Sw - Tw\| &= \|\widetilde{S}_1(w, 0) - Tw\| \\
&\leq \max \{ \|\widetilde{S}_1(w, 0) - Tw\|, \|\widetilde{S}_2(w, 0)\| \} \\
&= \|\widetilde{S}(w, 0) - \widetilde{T}(w, 0)\|_{\infty}
\leq \|\widetilde{S}(w, 0) - \widetilde{T}(w, 0)\|_a < \eps.
\end{align*}
So, $\|S - T\| < \eps$. It remains to prove that $w_1^*(w_1) = 1$ and $|w_1^*(Sw_1)| = 1$. To do so, we divide the proof in two cases.

\noindent
{\it Case 1:} Suppose that $\oplus_a$ is of type $\infty$. We will prove that $z_1^* = 0$. To do so, suppose that it is not true. By Lemma \ref{charac}.d, $\oplus_{a^*}$ is of type $1$. So, there is $K > 0$ such that
\begin{equation*}
\|w_1^*\| + K \|z_1^*\| \leq |(\|w_1^*\|, \|z_1^*\|)|_{a^*} = \|(w_1^*, z_1^*)\|_{a^*} = 1.
\end{equation*}
Notice that, since $\|(w_1, z_1) - (w_0, 0)\|_a < \eps$, we have that $\|z_1\| < \eps$. So, if $\eps \in (0, K)$, then
\begin{align*}
1 = w_1^*(w_1) + z_1^*(z_1) &\leq \|w_1^*\|\|w_1\| + \|z_1^*\|\|z_1\| \\
&< \|w_1^*\| + \eps \|z_1^*\| < \|w_1^*\| + K \|z_1^*\| \leq 1,
\end{align*}
which is a contradiction. So, $z_1^* = 0$. This implies that $(w_1, w_1^*) \in \Pi(W)$.

Now, by Lemma \ref{charac}.b, there is $b_0 > 0$ such that $|(1, b_0)|_a = 1$. Put $\rho = \frac{b_0}{\eps} > 0$. Then, $\|\rho z_1\| = \rho \|z_1\| < \rho \eps = b_0$, so 
\begin{equation*}
\|(w_1, \rho z_1)\|_a = |(\|w_1\|, \|\rho z_1\|)|_a \leq |(1, b_0)|_a = 1.
\end{equation*} 
Writing
\begin{equation*}
(w_1, z_1) = \left(1 - \frac{\eps}{b_0}\right) (w_1, 0) + \frac{\eps}{b_0} (w_1, \rho z_1)
\end{equation*}
we get that
\begin{equation*}
1 = \|(w_1, z_1)\|_a \leq \left(1 - \frac{\eps}{b_0} \right) \|(w_1, 0)\|_a + \frac{\eps}{b_0}\|(w_1, \rho z_1)\|_a \leq 1,
\end{equation*}
which implies that $\|(w_1, \rho z_1)\|_a = 1$. So, $((w_1, \rho z_1), (w_1^*, 0)) \in \Pi(W \oplus_a Z)$ and
\begin{equation*}
|\langle \widetilde{S}(w_1, \rho z_1), (w_1^*, 0) \rangle| \leq v(\widetilde{S}) = 1.
\end{equation*}
Therefore,
\begin{align*}
1 &= |\langle \widetilde{S}(w_1, z_1), (w_1^*, 0) \rangle| \\
&= \left| \left\langle \widetilde{S} \left(	\left(1 - \frac{\eps}{b_0}\right) (w_1, 0) + \frac{\eps}{b_0} (w_1, \rho z_1) \right), (w_1^*, 0) \right\rangle \right| \\
&\leq \left(1 - \frac{\eps}{b_0}\right) | \langle \widetilde{S}(w_1, 0), (w_1^*, 0) \rangle| + \frac{\eps}{b_0} | \langle \widetilde{S}(w_1, \rho z_1), (w_1^*, 0) \rangle| \leq 1.
\end{align*}
This shows that $|w_1^*(Sw_1)| = | \langle \widetilde{S}(w_1, 0), (w_1^*, 0) \rangle| = 1$. By (\ref{eq1}), we get that $v(S) = |w_1^*(Sw_1)| = 1$.

\noindent
{\it Case 2:} Now suppose that $\oplus_a$ is of type $1$. Since $\|(x_0^*, y_0^*) - (x, 0)\|_{a^*} < \eps$, we have that $\|y_0^*\| < \eps$. We will prove that $z_1 = 0$. Suppose not. By Lemma \ref{charac}.a, there is $K > 0$ such that
\begin{equation*}
\|w_1\| + K \|z_1\| \leq | (\|w_1\|, \|z_1\|)|_a = \|(w_1, z_1)\|_a = 1.
\end{equation*}
Therefore, if $\eps \in (0, K)$, then
\begin{align*}
1 = w_1^*(w_1) + z_1^*(z_1) &\leq \|w_1^*\|\|w_1\| + \|z_1^*\|\|z_1\| \\
&< \|w_1\| + \eps \|z_1\| < \|w_1\| + K \|z_1\| \leq 1,
\end{align*}
which is a contradiction. So, $z_1 = 0$ and $(w_1, w_1^*) \in \Pi(W)$.

By Lemma \ref{charac}.c, $\oplus_{a^*}$ is of type $\infty$. So, there is $b_0 > 0$ such that $|(1, b_0)|_{a^*} = 1$. Put $\rho = \frac{b_0}{\eps} > 0$. Then, $\|\rho z_1^*\| = \rho \|z_1^*\| < \rho \eps = b_0$. So,
\begin{equation*}
\|(w_1^*, \rho z_1^*)\|_{a^*} = |(\|w_1^*\|, \|\rho z_1^*\|)|_{a^*} \leq |(1, b_0)|_{a^*} = 1.
\end{equation*}
Write
\begin{equation*}
(w_1^*, z_1^*) = \left( 1 - \frac{\eps}{b_0} \right) (w_1^*, 0) + \frac{\eps}{b_0} (w_1^*, \rho z_1^*).
\end{equation*}
So, $\|(w_1^*, \rho z_1^*)\|_{a^*} = 1$ and then $((w_1, 0), (w_1^*, \rho z_1^*)) \in \Pi(W \oplus_a Z)$. This implies that
\begin{equation*}
| \langle \widetilde{S} (w_1, 0), (w_1^*, \rho z_1^*) \rangle| \leq v(\widetilde{S}) = 1
\end{equation*}
and
\begin{equation*}
1 = | \langle \widetilde{S}(w_1, 0), (w_1^*, z_1^*) \rangle|
\leq \left(1 - \frac{\eps}{b_0} \right) |(w_1^*, 0)(\widetilde{S}(w_1, 0))| + \frac{\eps}{b_0} |(w_1^*, \rho z_1^*)(\widetilde{S}(w_1, 0))| \leq 1.	
\end{equation*}
So, $|w_1^*(Sw_1)| = | \langle \widetilde{S}(w_1, 0), (w_1^*, 0) \rangle| = 1$ and this finishes the proof.
\end{proof}

We would like to point out that the same proof of Theorem \ref{numericalradius1} works also for compact operators and we have the analogous result for the BPBp-nu for compact operators.

\begin{prop} \label{numericalradius2} Let $X$ be a Banach space and let $W$ be an absolute summand of type 1 or $\infty$ of $X$. If $X$ has the BPBp-nu for compact operators, so does $W$.		
\end{prop}

Arguing as in Theorem \ref{numericalradius1}, we get the following result for the weak BPBp-nu. Recall that the only difference between the BPBp-nu and the weak BPBp-nu is the normalization of the numerical radius of the operator by its numerical radius (see Definition \ref{defBPBp-nu}). We include an sketch of the proof for completeness.

\begin{prop} \label{weaknumericalradius1} Let $X$ be a Banach space and let $W$ be an absolute summand of type 1 or $\infty$ of $X$. If $X$ has the weak BPBp-nu, so does $W$.	
\end{prop}

\begin{proof} Let $\eps \in (0, 1)$ be given and consider the weak BPBp-nu function $\eta(\eps) > 0$ for the Banach space $X = W \oplus_a Z$. Let $T \in \mathcal{L}(W)$ with $v(T) = 1$ and $(w_0, w_0^*) \in \Pi(W)$ be such that $|w_0^*(Tw_0)| > 1 - \eta(\eps)$. Using the same notation of Theorem \ref{numericalradius1} and applying our hypothesis, there are $\widetilde{S} \in \mathcal{L}(W \oplus_a Z)$  and $((w_1, z_1), (w_1^*, z_1^*)) \in \Pi(W \oplus_a Z)$ such that $|\langle \widetilde{S}(w_1, z_1), (w_1^*, z_1^*) \rangle | = v(\widetilde{S})$, $\|(w_1^*, z_1^*) - (w_0^*, 0)\|_{a^*} < \eps$, $\|(w_1, z_1) - (w_0, 0)\|_a < \eps$, and $\|\widetilde{S} - \widetilde{T}\| < \eps$. Then, $\|w_1^* - w_0^*\| < \eps$ and $\|w_1 - w_0\| < \eps$. Writing again $\widetilde{S} = (\widetilde{S}_1, \widetilde{S}_2)$, we define $S \in \mathcal{L}(W)$ by
$S(w) := \widetilde{S}_1(w, 0)$ for every $w \in W$. Then, $v(S) \leq v(\widetilde{S})$ and $\|S - T\| < \eps$.

Suppose that $\oplus_a$ is of type $\infty$. As in Theorem \ref{numericalradius1} (Case 1), we have that $z_1^* = 0$ and then $(w_1, w_1^*) \in \Pi(W)$. Since $\oplus_a$ is of type $\infty$, there is $b_0 > 0$ such that $|(1, b_0)|_a = 1$. Put $\rho = \frac{b_0}{\eps} > 0$. Since $\|z_1\| < \eps$, we have that $\|(w_1, \rho z_1)\|_a \leq |(1, b_0)|_a = 1$. Then, $((w_1, \rho z_1), (w_1^*, 0)) \in \Pi(W \oplus_a Z)$. This implies that $|\langle \widetilde{S}(w_1, \rho z_1), (w_1^*, 0) \rangle| \leq v(\widetilde{S})$. Therefore,
\begin{align*}
	v(\widetilde{S}) &= |\langle \widetilde{S}(w_1, z_1), (w_1^*, 0) \rangle| \\
	&= \left| \left\langle \widetilde{S} \left(	\left(1 - \frac{\eps}{b_0}\right) (w_1, 0) + \frac{\eps}{b_0} (w_1, \rho z_1) \right), (w_1^*, 0) \right\rangle \right| \\
	&\leq \left(1 - \frac{\eps}{b_0}\right) | \langle \widetilde{S}(w_1, 0), (w_1^*, 0) \rangle| + \frac{\eps}{b_0} | \langle \widetilde{S}(w_1, \rho z_1), (w_1^*, 0) \rangle| \\
	&= \left(1 - \frac{\eps}{b_0}\right) |w_1^*(Sw_1)| + \frac{\eps}{b_0} | \langle \widetilde{S}(w_1, \rho z_1), (w_1^*, 0) \rangle| \leq v(\widetilde{S}).
\end{align*}
This shows that $|w_1^*(Sw_1)| = v(\widetilde{S}) \geq v(S)$. So, $|w_1^*(Sw_1)| = v(S)$ and $W$ has the weak BPBp-nu.

Now suppose that $\oplus_a$ is of type 1. As in Theorem \ref{numericalradius1} (Case 2), $z_1 = 0$ and then $(w_1, w_1^*) \in \Pi(W)$. Since $\oplus_{a^*}$ is of type $\infty$ (see Lemma \ref{charac}.(c)), there is $b_0 > 0$ such that $|(1, b_0)|_{a^*} = 1$. Put $\rho = \frac{b_0}{\eps} > 0$. Since $\|z_1^*\| < \eps$, we have that $\|(w_1^*, \rho z_1^*)\|_{a^*} \leq |(1, b_0)|_{a^*} = 1$. Then $((w_1, 0), (w_1^*, \rho z_1^*)) \in \Pi(W \oplus_a Z)$. This implies that $|\langle \widetilde{S}(w_1, 0), (w_1^*, \rho z_1^*) \rangle | \leq v(\widetilde{S})$. Therefore,
\begin{align*}
v(\widetilde{S}) &= | \langle \widetilde{S} (w_1, 0), (w_1^*, z_1^*) \rangle| \\
&= \left| \left \langle \widetilde{S}(w_1, 0), \left( 1 - \frac{\eps}{b_0} \right) (w_1^*, 0) + \frac{\eps}{b_0} (w_1^*, \rho z_1^*) \right \rangle \right| \\
&\leq \left(1 - \frac{\eps}{b_0} \right) | \langle  \widetilde{S}(w_1, 0), (w_1^*, 0) \rangle| + \frac{\eps}{b_0} | \langle \widetilde{S}(w_1, 0), (w_1^*, \rho z_1^* ) \rangle | \\
&= \left(1 - \frac{\eps}{b_0} \right) | w_1^*(Sw_1)| + \frac{\eps}{b_0} | \langle \widetilde{S}(w_1, 0), (w_1^*, \rho z_1^* ) \rangle | \leq v(\widetilde{S}).
\end{align*} 	
This shows that $|w_1^*(Sw_1)| = v(\widetilde{S}) \geq v(S)$. So, $|w_1^*(Sw_1)| = v(S)$ and $W$ has the weak BPBp-nu.	
\end{proof}

Again, the above proof can be adapted to compact operators to get the following result.

\begin{prop} Let $X$ be a Banach space and let $W$ be an absolute summand of type 1 or $\infty$ of $X$. If $X$ has the weak BPBp-nu for compact operators, so does $W$.	
\end{prop}

Next, we have interest to investigate the density of the set of numerical radius attaining operators. To do so, we will prove the following easy lemma which says that, in order to prove the denseness of the set $\NRA(X)$ for a Banach space $X$, it is enough to consider operators with numerical radius one. Observe that $T \in \mathcal{L}(X)$ attains its numerical radius if and only if the operator $\lambda T$ does for every $\lambda \in \R$.

\begin{lemma} \label{lemmaNRA} Let $X$ be a Banach space. The following statements are equivalents.
\begin{itemize}
\item[(a)] The set $\NRA(X)$ is dense in $\mathcal{L}(X)$.
\item[(b)] For every $T \in \mathcal{L}(X)$ with $v(T) = 1$, there is a sequence $\{S_n\} \subset \NRA(X)$ with $v(S_n) = 1$ for every $n \in \N$ and such that $S_n \longrightarrow T$ in norm.
\item[(c)] For every $T \in \mathcal{L}(X)$ with $v(T) = 1$ and every $\eps>0$, there is $S \in \mathcal{L}(X)$ with $v(S) = 1$ such that $\|S - T\| < \eps$.	
\end{itemize}	
\end{lemma}

\begin{proof} (a) $\Rightarrow$ (b). Let $T \in \mathcal{L}(X)$ with $v(T) = 1$. By hypothesis, there is a sequence $\{S_n'\} \subset \NRA(X)$ such that $S_n' \longrightarrow T$ in norm. This implies that there is $n_0 \in \N$ such that $|v(S_n') - v(T)| \leq \|S_n' - T\| < 1$ for all $n \geq n_0$. Since $v(T) = 1$, we have that $v(S_n') > 0$ for all $n \geq n_0$. Consider then the sequence $(S_n)$ defined by $$S_n := \frac{1}{v(S_{n+n_0}')}S_{n+n_0}' \in \mathcal{L}(X) \qquad (n\in \N).$$ Then $\{S_n\} \subset \NRA(X)$, $v(S_n) = 1$ for all $n \in \N$ and, since $v(S_n') \longrightarrow v(T) = 1$, $S_n \longrightarrow T$ in norm. This proves (b).
	
\noindent
(b) $\Leftrightarrow$ (c) is immediate.

\noindent
(b) $\Rightarrow$ (a). Let $T \in \mathcal{L}(X)$ be given. If $v(T) = 0$, then $T$ attains its numerical radius and we are done. Otherwise, $v(T) \not= 0$ and we may consider the operator $T' := \frac{T}{v(T)}$ which satisfies that $v(T') = 1$. By hypothesis, there is a sequence $\{S_n\} \subset \NRA(X)$ with $v(S_n) = 1$ for all $n \in \N$ and such that $S_n \longrightarrow T'$ in norm. This implies that $v(T) S_n \longrightarrow v(T)T'=T$ and we are done since $v(T) S_n \in \NRA(X)$ for every $n \in \N$.
\end{proof}

Now we are ready to provide a result for the denseness of the operators which attain their numerical radius.

\begin{prop} \label{NRA1} Let $X$ be a Banach space and let $W$ be an absolute summand of $X$ of type 1 or $\infty$. If $\NRA(X)$ is dense in $\mathcal{L}(X)$, then $\NRA(W)$ is dense in $\mathcal{L}(W)$.
\end{prop}

\begin{proof} Let $Z$ be a Banach space such that $X = W \oplus_a Z$. Let $\eps \in (0, 1)$ and $T \in \mathcal{L}(W)$ be given. We may consider $v(T) = 1$ by using Lemma \ref{lemmaNRA}. Define $\widetilde{T} \in \mathcal{L}(W \oplus_a Z)$ by $\widetilde{T}(w, z) := (Tw, 0)$ for every $(w, z) \in W \oplus_a Z$. By Lemma \ref{radius}, $v(\widetilde{T}) = v(T) = 1$. Since $\NRA(W \oplus_a Z)$ is dense in $\mathcal{L}(W \oplus_a Z)$, there are $\widetilde{S} \in \mathcal{L}(W \oplus_a Z)$ with $v(\widetilde{S}) = 1$ and $((w_0, z_0), (w_0^*, z_0^*)) \in \Pi(W \oplus_a Z)$ such that
\begin{equation*}
|\langle \widetilde{S}(w_0, z_0), (w_0^*, z_0^*) \rangle| = v(\widetilde{S}) = 1 \ \ \ \mbox{and} \ \ \ \|\widetilde{S} - \widetilde{T}\| < \eps.	
\end{equation*}
Set $\widetilde{S} = (\widetilde{S}_1, \widetilde{S}_2)$, where $\widetilde{S}_1: W \oplus_a Z \longrightarrow W$ and $\widetilde{S}_2: W \oplus_a Z \longrightarrow Z$. By Lemma \ref{state0}, we have
\begin{equation} \label{state}
w_0^*(w_0) = \|w_0^*\| \|w_0\| \ \ \ \mbox{and} \ \ \ z_0^*(z_0) = \|z_0^*\|\|z_0\|.	
\end{equation}

\noindent
Now for all $(w, z) \in B_{W \oplus_a Z}$, we have that
\begin{align*}
\max \{ \| \widetilde{S}_1 (w, z) - Tw \|, \| \widetilde{S}_2 (w, z)\| \} &= \| ( \widetilde{S}_1(w, z) - Tw, \widetilde{S}_2(w, z))\|_{\infty} \\
&= \|((\widetilde{S}_1(w, z), \widetilde{S}_2(w, z)) - (Tw, 0))\|_{\infty} \\
&=\| \widetilde{S}(w, z) - \widetilde{T}(w, z)\|_{\infty} \\
&\leq\| \widetilde{S}(w, z) - \widetilde{T}(w, z)\|_a
\leq \|\widetilde{S} - \widetilde{T} \| < \eps.	
\end{align*}
This implies that, for all $w \in B_W$,
\begin{equation} \label{ineq1}
	\|\widetilde{S}_1(w, 0) - Tw\| < \eps \ \ \ \mbox{and} \ \ \ \|\widetilde{S}_2\| < \eps.
\end{equation}
On the other hand, for all $z \in B_Z$, we get that
\begin{equation} \label{ineq}
	\|\widetilde{S}(0, z)\|_a = \|\widetilde{S} (0, z) - \widetilde{T}(0, z)\|_a \leq \|\widetilde{S} - \widetilde{T}\| < \eps.
\end{equation}
In particular, $\|\widetilde{S}_1(0, z)\| < \eps$ and $\|\widetilde{S}_2(0, z)\| < \eps$ for all $z \in B_Z$.

\noindent
{\it Claim:} $w_0 \not= 0$.
Otherwise, by using (\ref{ineq}), we would have
\begin{equation*}
1 = | \langle \widetilde{S}(0, z_0), (w_0^*, z_0^*) \rangle| \leq \|\widetilde{S}(0, z_0)\|_a < \eps,	
\end{equation*}
which is a contradiction.

\noindent
{\it Claim:} $w_0^* \not= 0$.
Otherwise, by using (\ref{ineq1}), we would have
\begin{equation*}
	1 = | \langle \widetilde{S}(w_0, z_0), (0, z_0^*) \rangle| = |z_0^*(\widetilde{S}_2(w_0, z_0))| \leq \| \widetilde{S}_2\| < \eps,	
\end{equation*}	
which is a new contradiction.	
	
\vspace{0.3cm}

Therefore, since $w_0, w_0^* \not=0$, by using (\ref{state}), we have that
\begin{equation} \label{state1}
\left( \left(\frac{w_0}{\|w_0\|}, 0 \right),  \left(\frac{w_0^*}{\|w_0^*\|}, 0 \right) \right) \in \Pi(W \oplus_a Z).
\end{equation}
Define the operator $S \in \mathcal{L}(W)$ by
\begin{equation*}
S(w) := \widetilde{S}_1(w, 0) \ \ \ (w \in W).
\end{equation*}
Note that for all $(w, w^*) \in \Pi(W)$, we have that $((w, 0), (w^*, 0)) \in \Pi(W \oplus_a Z)$ and then,
\begin{equation*}
|w^*(Sw)| = |w^*(\widetilde{S}_1(w, 0))| = | \langle \widetilde{S}(w, 0), (w^*, 0) \rangle| \leq v(\widetilde{S}) = 1.
\end{equation*}
This shows that $v(S) \leq 1$. Also, by using (\ref{ineq1}), note that for all $w \in S_W$,
\begin{equation*}
\|Sw - Tw\| = \|\widetilde{S}_1(w, 0) - T(w)\| < \eps.	
\end{equation*}
So, $\|S - T\| < \eps$. It remains to prove that $S$ attains its numerical radius, and we do this separating the proof in two cases.

\noindent
{\it Case 1:} Assume first that $\oplus_a$ is of type $1$. We will prove that $z_0 = 0$. Suppose not. Since $\oplus_a$ is of type $1$, by Lemma \ref{charac}.a there is $K > 0$ such that $\|w_0\| + K\|z_0\| \leq \|(w_0, z_0)\|_a = 1$. On the other hand, being $\oplus_{a^*}$ of type $\infty$ (see Lemma \ref{charac}.c), there is $b_0 > 0$ such that $|(1, b_0)|_{a^*} = 1$ and then
\begin{equation*}
\left\| \left( \frac{w_0^*}{\|w_0^*\|}, b_0 z_0^* \right) \right\|_{a^*} \leq |(1, b_0)|_{a^*} = 1.
\end{equation*}
Using (\ref{state}), we have that
\begin{equation*}
\left \langle \left( \frac{w_0}{\|w_0\|}, 0 \right), \left( \frac{w_0^*}{\|w_0^*\|}, b_0 z_0^* \right) \right \rangle = 1.
\end{equation*}
Then,
\begin{equation*}
\left\| \left( \frac{w_0^*}{\|w_0^*\|}, b_0 z_0^* \right) \right\|_{a^*} = 1 \ \ \ \mbox{and} \ \ \ \left(\left( \frac{w_0}{\|w_0\|}, 0 \right), \left( \frac{w_0^*}{\|w_0^*\|}, b_0 z_0^* \right)\right) \in \Pi(W \oplus_a Z).
\end{equation*}
This implies that
\begin{equation} \label{NRAineq}
	\left| \left \langle \widetilde{S} \left( \frac{w_0}{\|w_0\|}, 0 \right), \left( \frac{w_0^*}{\|w_0^*\|}, b_0 z_0^* \right) \right \rangle \right| \leq v(\widetilde{S}) = 1.
\end{equation}
Now, set
\begin{equation*}
(w_0^*, z_0^*) = \left(1 - \frac{1}{b_0\|w_0^*\|} \right)(w_0^*, 0) + \frac{1}{b_0 \|w_0^*\|} (w_0^*, b_0 \|w_0^*\| z_0^*).	
\end{equation*}
So, we have that
\begin{align*}
|\langle \widetilde{S}(w_0, 0), (w_0^*, z_0^*) \rangle| &= \left| \left \langle \widetilde{S}(w_0, 0), \left( 1 - \frac{1}{b_0 \|w_0^*\|} \right) (w_0^*, 0) + \frac{1}{b_0 \|w_0^*\|} (w_0^*, b_0 \|w_0^*\| z_0^*) \right \rangle \right| \\
&\leq \left(1 - \frac{1}{b_0 \|w_0^*\|} \right) |\langle \widetilde{S}(w_0, 0), (w_0^*, 0) \rangle| + \frac{1}{b_0 \|w_0^*\|} |\langle \widetilde{S}(w_0, 0), (w_0^*, b_0 \|w_0^*\| z_0^*) \rangle|
\end{align*}
By using (\ref{state1}), we get that
\begin{align*}
\left(1 - \frac{1}{b_0 \|w_0^*\|} \right) | \langle \widetilde{S}(w_0, 0), (w_0^*, 0) \rangle| &= \left(1 - \frac{1}{b_0 \|w_0^*\|} \right) \|w_0^*\| \|w_0\| \left| \left \langle \widetilde{S} \left( \frac{w_0}{\|w_0\|}, 0 \right), \left( \frac{w_0^*}{\|w_0^*\|}, 0 \right) \right \rangle \right| \\ &\leq \left(1 - \frac{1}{b_0 \|w_0^*\|}\right) \|w_0\| v(\widetilde{S}) = \left(1 - \frac{1}{b_0 \|w_0^*\|}\right) \|w_0\|.
\end{align*}
Now, using (\ref{NRAineq}),
\begin{align*}
\frac{1}{b_0 \|w_0^*\|} | \langle \widetilde{S}(w_0, 0), (w_0^*, b_0 \|w_0^*\| z_0^*) \rangle | &= \frac{1}{b_0 \|w_0^*\|} \|w_0^*\|\|w_0\| \left| \left \langle \widetilde{S} \left( \frac{w_0}{\|w_0\|}, 0 \right), \left( \frac{w_0^*}{\|w_0^*\|}, \frac{b_0 \|w_0^*\| z_0^*}{\|w_0^*\|} \right) \right \rangle \right| \\
&= \frac{1}{b_0\|w_0^*\|} \|w_0^*\| \|w_0\| \left| \left \langle \widetilde{S} \left( \frac{w_0}{\|w_0\|}, 0 \right), \left( \frac{w_0^*}{\|w_0^*\|}, b_0 z_0^*\right) \right \rangle \right| \\
&\leq \frac{1}{b_0 \|w_0^*\|} v(\widetilde{S} )\|w_0\| = 	\frac{1}{b_0 \|w_0^*\|} \|w_0\|
\end{align*}
Then,
\begin{equation*}
| \langle \widetilde{S}(w_0, 0), (w_0^*, z_0^*) \rangle| \leq \left(1 - \frac{1}{b_0 \|w_0^*\|}\right) \|w_0\| + \frac{1}{b_0\|w_0^*\|}\|w_0\| = \|w_0\|.	
\end{equation*}
On the other hand, by using (\ref{ineq}), we have that
\begin{align*}
| \langle \widetilde{S}(0, z_0), (w_0^*, z_0^*) \rangle| &= \|z_0\|   \left| \left \langle \widetilde{S} \left(0, \frac{z_0}{\|z_0\|} \right), (w_0^*, z_0^*) \right \rangle \right|	\\
&\leq 	\|z_0\| \left\| \widetilde{S} \left(0, \frac{z_0}{\|z_0\|} \right) \right\| < \eps \|z_0\|.
\end{align*}
Using these inequalities, for all $\eps \in (0, K)$, we get that
\begin{align*}
1 &= | \langle \widetilde{S}(w_0, z_0), (w_0^*, z_0^*) \rangle| \\
&\leq |\langle \widetilde{S}(w_0, 0), (w_0^*, z_0^*) \rangle| + |\langle \widetilde{S}(0, z_0), (w_0^*, z_0^*) \rangle | \\
&< \|w_0\| + \eps \|z_0\| < \|w_0\| + K \|z_0\| \leq 1,
\end{align*}
which is a contradiction. So, $z_0 = 0$.

\vspace{0.3cm}

\noindent
Being $z_0 = 0$, we have that $(w_0, w_0^*) \in \Pi(W)$ and $| \langle \widetilde{S}(w_0, 0), (w_0^*, z_0^*) \rangle | = 1$. Let us show that $v(S)=1$ and that it is attained. Indeed, since $(w_0, w_0^*) \in \Pi(W)$, we have that
\begin{equation*}
	| \langle \widetilde{S}(w_0, 0), (w_0^*, b_0 z_0^*) \rangle | \leq v(\widetilde{S}) \leq 1.	
\end{equation*}
Then,
\begin{align*}
1 = | \langle \widetilde{S}(w_0, 0), (w_0^*, z_0^*) \rangle|
&= \left| \left( 1 - \frac{1}{b_0} \right) \langle \widetilde{S}(w_0, 0), (w_0^*, 0) \rangle + \frac{1}{b_0} \langle \widetilde{S}(w_0, 0), (w_0^*, b_0 z_0^*) \rangle \right| \\
&\leq 	 \left( 1 - \frac{1}{b_0} \right) | \langle \widetilde{S}(w_0, 0), (w_0^*, 0) \rangle| + \frac{1}{b_0} |\langle \widetilde{S}(w_0, 0), (w_0^*, b_0 z_0^*) \rangle | \\
&= \left( 1 - \frac{1}{b_0} \right) |  w_0^*(\widetilde{S}_1(w_0, 0)) | + \frac{1}{b_0} |\langle \widetilde{S}(w_0, 0), (w_0^*, b_0 z_0^*) \rangle | \\
&= \left( 1 - \frac{1}{b_0} \right) | w_0^*(Sw_0)| + \frac{1}{b_0} |\langle \widetilde{S}(w_0, 0), (w_0^*, b_0 z_0^*) \rangle | \\
&\leq \left(1 - \frac{1}{b_0} \right) v(S) + \frac{1}{b_0} v(\widetilde{S}) \leq 1.
\end{align*}
This implies that $|w_0^*(Sw_0)| = v(S) = 1$.

\noindent
{\it Case 2:} Now assume that $\oplus_a$ is of type $\infty$. We will prove that $z_0^* = 0$. Suppose not. By Lemma \ref{charac}.b, there is $b_0 > 0$ such that $|(1, b_0)|_a = 1$. Then,
\begin{equation*}
\left\| \left(\frac{w_0}{\|w_0\|}, b_0 z_0 \right) \right\|_a \leq |(1, b_0)|_a = 1.
\end{equation*}
By (\ref{state1}), we have that
\begin{equation*}
\left \langle \left( \frac{w_0}{\|w_0\|}, b_0 z_0 \right), \left( \frac{w_0^*}{\|w_0^*\|} , 0 \right) \right \rangle = 1	
\end{equation*}
So,
\begin{equation*}
\left\| \left(\frac{w_0}{\|w_0\|}, b_0 z_0 \right) \right\|_a = 1 \qquad \text{and} \qquad \left( \left( \frac{w_0}{\|w_0\|}, b_0 z_0 \right), \left( \frac{w_0^*}{\|w_0^*\|} , 0 \right) \right) \in \Pi(W \oplus_a Z).
\end{equation*}
This implies that
\begin{equation*}
\left| \left \langle \widetilde{S} \left( \frac{w_0}{\|w_0\|}, b_0 z_0 \right), \left( \frac{w_0^*}{\|w_0^*\|}, 0 \right) \right \rangle \right| \leq v(\widetilde{S}) = 1.	
\end{equation*}

\noindent
Since $\oplus_{a^*}$ is of type $1$ (see Lemma \ref{charac}.d), there is $K > 0$ such that $\|w_0^*\| + K \|z_0^*\| \leq \|(w_0^*, z_0^*) \|_{a^*} = 1$. Set
\begin{equation*}
(w_0, z_0) = \left(1 - \frac{1}{b_0 \|w_0\|} \right) (w_0, 0) + \frac{1}{b_0 \|w_0\|} (w_0, b_0 \|w_0\| z_0).	
\end{equation*}
Then,
\begin{align*}
1 &= | \langle \widetilde{S}(w_0, z_0), (w_0^*, z_0^*) \rangle | \\
&= | \langle \widetilde{S}(w_0, z_0), (w_0^*, 0) \rangle + \langle \widetilde{S}(w_0, z_0), (0, z_0^*) \rangle | \\
&\leq \left|  \left \langle \widetilde{S} \left( \left( 1 - \frac{1}{b_0 \|w_0\|} \right) (w_0, 0) + \frac{1}{b_0 \|w_0\|}  (w_0, b_0 \|w_0\| z_0) \right), (w_0^*, 0) \right \rangle \right| + |\langle \widetilde{S}(w_0, z_0), (0, z_0^*) \rangle | \\
&\leq \left( 1 - \frac{1}{b_0 \|w_0\|} \right) | \langle \widetilde{S} (w_0, 0), (w_0^*, 0) \rangle| + \frac{1}{b_0 \|w_0\|} | \langle \widetilde{S}(w_0, b_0 \|w_0\| z_0), (w_0^*, 0) \rangle| + |\langle \widetilde{S}(w_0, z_0), (0, z_0^*) \rangle |
\end{align*}
But we have that
\begin{align*}
\left( 1 - \frac{1}{b_0 \|w_0\|} \right) | \langle \widetilde{S} (w_0, 0), (w_0^*, 0) \rangle| &= \left( 1 - \frac{1}{b_0 \|w_0\|} \right) \|w_0^*\|\|w_0\| \left| \left\langle \widetilde{S} \left( \frac{w_0}{\|w_0\|}, 0\right), \left(\frac{w_0^*}{\|w_0^*\|}, 0 \right) \right\rangle \right| \\
&\leq\left( 1 - \frac{1}{b_0 \|w_0\|} \right) \|w_0^*\|
\end{align*}
and
\begin{align*}
\frac{1}{b_0 \|w_0\|} | \langle \widetilde{S}(w_0, b_0 \|w_0\| z_0), (w_0^*, 0) \rangle | &= \frac{1}{b_0 \|w_0\|} \|w_0^*\|\|w_0\| \left| \left \langle \widetilde{S} \left( \frac{w_0}{\|w_0\|}, \frac{b_0 \|w_0\| z_0}{\|w_0\|} \right), \left( \frac{w_0^*}{\|w_0^*\|}, 0 \right) \right \rangle \right| \\
&= \frac{1}{b_0 \|w_0\|} \|w_0^*\|\|w_0\| \left| \left \langle \widetilde{S} \left( \frac{w_0}{\|w_0\|}, b_0 z_0 \right), \left( \frac{w_0^*}{\|w_0^*\|}, 0 \right) \right \rangle \right| \\
&\leq	\frac{1}{b_0 \|w_0\|}\|w_0^*\| v(\widetilde{S}) = \frac{1}{b_0 \|w_0\|}\|w_0^*\|.
\end{align*}
Moreover, by using (\ref{ineq1}), we have that
\begin{equation*}
| \langle \widetilde{S}(w_0, z_0), (0, z_0^*) \rangle| = \|z_0^*\| 	\left| \left\langle \widetilde{S}(w_0, z_0), \left(0, \frac{z_0^*}{\|z_0^*\|} \right) \right\rangle \right| = \|z_0^*\| \left| \left(\frac{z_0^*}{\|z_0^*\|}\right) \widetilde{S}_2(w_0, z_0) \right|  \leq \|z_0^*\| \|\widetilde{S}_2\| < \eps \|z_0^*\|.
\end{equation*}
Therefore, for all $\eps \in (0, K)$, we have
\begin{align*}
1 &\leq \left( 1 - \frac{1}{b_0 \|w_0\|} \right) | \langle \widetilde{S} (w_0, 0), (w_0^*, 0) \rangle| + \frac{1}{b_0 \|w_0\|} | \langle \widetilde{S}(w_0, b_0 \|w_0\| z_0), (w_0^*, 0) \rangle| + |\langle \widetilde{S}(w_0, z_0), (0, z_0^*) \rangle | \\
&\leq \left(1 - \frac{1}{b_0 \|w_0\|} \right) \|w_0^*\| + \frac{1}{b_0 \|w_0\|} \|w_0^*\| + \eps \|z_0^*\|	\\
&= \|w_0^*\| + \eps \|z_0^*\| < \|w_0^*\| + K \|z_0^*\| \leq 1.
\end{align*}
This contradiction gives $z_0^* = 0$. So, $(w_0, w_0^*) \in \Pi(W)$ and
\begin{equation*}
1 = | \langle \widetilde{S}(w_0, z_0), (w_0^*, 0) \rangle| = |w_0^*(\widetilde{S}_1(w_0, z_0))|.
\end{equation*}
Now, since $((w_0, b_0 z_0), (w_0^*, 0)) \in \Pi(W)$, we have that
\begin{align*}
1 &= | \langle \widetilde{S}(w_0, z_0), (w_0^*, 0) \rangle | \\
&\leq \left(1 - \frac{1}{b_0} \right) |\langle \widetilde{S}(w_0, 0), (w_0^*, 0) \rangle| + \frac{1}{b_0} | \langle \widetilde{S}(w_0, b_0 z_0), (w_0^*, 0) \rangle| \\
&\leq \left(1 - \frac{1}{b_0} \right) |w_0^*(Sw_0)| + \frac{1}{b_0} v(\widetilde{S}) \\
&\leq 	\left(1 - \frac{1}{b_0} \right) v(S) + \frac{1}{b_0} v(\widetilde{S}) \leq 1.
\end{align*}
So, $|w_0^*(Sw_0)| = v(S) = 1$ and this finishes the proof.
\end{proof}

As far as we know, Proposition \ref{NRA1} is new even for $L$- and $M$-summands.

\begin{cor} \label{NR1} Let $X$ be Banach space.
\begin{itemize}	
\item[(a)] If $W$ is an $L$-summand of $X$ and $\overline{\NRA(X)} = \mathcal{L}(X)$, then $\overline{\NRA(W)} = \mathcal{L}(W)$
\item[(b)] If $W$ is an $M$-summand of $X$ and $\overline{\NRA(X)} = \mathcal{L}(X)$, then $\overline{\NRA(W)} = \mathcal{L}(W)$
\end{itemize}
\end{cor}

Adapting Proposition \ref{NRA1} to the compact case, we have the following result.

\begin{prop} \label{NRAK}  Let $X$ be a Banach space and let $W$ be an absolute summand of $X$ of type 1 or $\infty$. If $\NRA(X) \cap \mathcal{K}(X)$ is dense in $\mathcal{K}(X)$, then $\NRA(W) \cap \mathcal{K}(W)$ is dense in $\mathcal{K}(W)$.
\end{prop}

Let us finally discuss on the validity of reciprocal results. In \cite[Remark 2.5]{KLM1} that the sets $\NRA(C[0, 1] \oplus_1 L_1[0, 1])$ and $\NRA(C[0, 1] \oplus_\infty L_1[0, 1])$ are not dense, respectively, in $\mathcal{L}(C[0, 1] \oplus_1 L_1[0, 1])$ and $\mathcal{L}(C[0, 1] \oplus_\infty L_1[0, 1])$ was observed. As both $C[0,1]$ and $L_1[0,1]$ have the BPBp for numerical radius \cite{AGR,KLM1}, this shows that there is no possible valid reciprocal results for Theorem \ref{numericalradius1} and Proposition \ref{NRA1}. On the other hand,  the sets $\NRA(C[0, 1] \oplus_1 L_1[0, 1]) \cap \mathcal{K}(C[0, 1] \oplus_{1} L_1[0, 1])$ and $\NRA(C[0, 1] \oplus_{\infty} L_1[0, 1]) \cap \mathcal{K}(C[0, 1] \oplus_{\infty} L_1[0, 1])$ are dense in $\mathcal{K}(C[0, 1] \oplus_1 L_1[0, 1])$ and $\mathcal{K}(C[0, 1] \oplus_{\infty} L_1[0, 1])$, respectively (see \cite[Example 3.4]{CapMM}). Nevertheless, we do not know if there is some reciprocal result for the BPBp-nu for compact operators.


\begin{thebibliography}{99}
	
\bibitem{Acosta-RACSAM}
\textsc{M.~D.~Acosta},
Denseness of norm attaining mappings, \emph{RACSAM} {\bf 100} (2006), 9--30.

\bibitem{AAGM} \textsc{M.~D.~Acosta, R.~ M.~Aron, D.~Garc\'ia, and M.~Maestre}, The Bishop-Phelps-Bollob\'as theorem for operators, \emph{J. Funct. Anal.} {\bf 294} (2008), 2780--2899.

\bibitem{AcoFakSole} \textsc{M.~D.~Acosta, M.~Fakhar, and M.~Soleimani-Mourchehkhorti}, The Bishop-Phelps-Bollob\'{a}s property for numerical radius of operators on $L_1(\mu)$, \emph{J. Math. Anal. Appl.} \textbf{458} (2018), 925--936.

\bibitem{AMS} \textsc{M.~D.~Acosta, M.~Masty{\l}o, and M.~Soleimani-Mourchehkhorti}, The Bishop-Phelps-Bollob\'as and approximate hyperplane series properties, \emph{J. Funct. Anal.} \textbf{274} (2018), no. 9, 2673--2699.
	
\bibitem{ACKLM} \textsc{R.~M.~Aron, Y.~S.~Choi, S.~K.~Kim, H.~J.~Lee, and M. Mart\'in}, The Bishop-Phelps-Bollob\'{a}s version of Lindenstrauss properties A and B, \emph{Trans. Amer. Math. Soc.} {\bf 367} (2015), 6085--6101.

\bibitem{AGR} \textsc{A.~Avil\'es, A.~ J.~Guirao, and J.~Rodr\'iguez}, On the Bishop-Phelps-Bollob\'as property for numerical radius in $C(K)$ spaces, \emph{J. Math. Anal. Appl.} {\bf 419} (2014), 395--421.
	
\bibitem{BP} \textsc{E.~Bishop and R.~R.~Phelps}, A proof that every Banach space is reflexive, \emph{Bull. Amer. Math. Soc.} {\bf 67} (1961), 97--98.
	
\bibitem{Bol}
\textsc{B.~Bollob\'{a}s}, An extension to the Theorem of Bishop and Phelps, \emph{Bull. London Math. Soc.} {\bf 2} (1970), 181--182.

\bibitem{Bonsall0} \textsc{F.~F.~Bonsall and J.~Duncan}, \emph{Numerical ranges of operators on normed spaces and of elements of normed algebras}, London Mathematical Society Lecture Note Series, 2. Cambridge University Press, London, New York (1971).

\bibitem{Bonsall1} \textsc{F.~F.~Bonsall and J.~Duncan}, \emph{Numerical ranges II}, London Mathematical Society Lecture Notes Series, No. 10.
Cambridge University Press, New York, London (1973). 

\bibitem{Bour} \textsc{J.~Bourgain}, On dentability and the Bishop-Phelps property, \emph{Israel J. Math.} \textbf{78} (1977), 265--271.

\bibitem{CapMM} \textsc{A.~Capel, M.~Mart\'in, and J.~Mer\'i}, Numerical radius attaining compact linear operators, \emph{J. Math. Anal. Appl.} \textbf{445}, 1258--1266.
	
\bibitem{CKMM} \textsc{M.~Chica, V.~Kadets, M.~Mart\'in, F.~Rambla-Barreno, and S.~Moreno-Pulido}, Bishop-Phelps-Bollob\'as modudi of a Banach space, \emph{J. Math. Anal. Appl.} {\bf 412} (2014), 697--719.

\bibitem{CMM} \textsc{M.~Chica, M.~Mart\'in, and J.~Mer\'i}, Numerical radius of rank-1 operators on Banach spaces, \emph{Quart. J. Math.} {\bf 65} (2014), 89--100.

\bibitem{Cho-Choi} \textsc{D.~H.~Cho and Y.~S.~Choi}, The Bishop-Phelps-Bollob\'{a}s theorem on bounded closed convex sets, \emph{J. Lond. Math. Soc.} \textbf{93} (2016),  502--518.

\bibitem{DGMM} \textsc{S.~Dantas, D.~Garc\'ia, M.~Maestre, and M.~Mart\'in}, The Bishop-Phelps-Bollob\'as property for compact operators, \emph{Can. J. Math.} \textbf{70} (2018), no. 1, 53–-73.

\bibitem{DKKLM} \textsc{S.~Dantas, V.~Kadets, S.~K.~Kim, H.~J.~Lee, and M.~Mart\'in}, On the pointwise Bishop--Phelps--Bollob\'as property for operators, accepted in \emph{Can. J. Math.} \href{https://doi.org/10.4153/S0008414X18000032}{doi:10.4153/S0008414X18000032}.

	\bibitem{DKL} \textsc{S.~Dantas, S.~K.~Kim, and H.~J.~Lee}, The Bishop-Phelps-Bollob\'as point property, \emph{J. Math. Anal. Appl.} \textbf{444} (2016), 1739--1751.
	
\bibitem{G} \textsc{F.~J.~Garc\'ia-Pacheco}, The AHSP is inherited by $E$-summands, \emph{ Adv. Oper. Theory} \textbf{2} (2017), 17--20.

\bibitem{GK} \textsc{A.~Guirao and O.~Kozhushkina}, The Bishop-Phelp-Bollob\'as property for numerical radius in $\ell_1(\C)$, \emph{Studia Math.} {\bf 218} (2013), 41--54.

\bibitem{Hardtke} \textsc{J.~Hardtke}, Absolute sums of Banach spaces and some geometric properties related to rontundity and smoothness,  \emph{Banach J. Math. Anal.} \textbf{8} (2014), 295--334.

\bibitem{HWW} \textsc{P.~Harmand, D.~Werner, and D.~Werner},
\emph{$M$-ideals in Banach spaces and Banach algebras}, Lecture Notes in Math. \textbf{1547}, Springer-Verlag, Berlin 1993.

\bibitem{KLM1} \textsc{S.~K.~Kim, H.~J.~Lee, and M.~Mart\'in}, On the Bishop-Phelps-Bollob\'as property for numerical radius, \emph{Abs. Appl. Anal.} vol.~2014, Article ID 479208, 15 pages, 2014.

\bibitem{KLM2} \textsc{S.~K.~Kim, H.~J.~Lee, M.~Mart\'in, and J. Mer\'i}, On the second numerical index for Banach spaces, accepted in \emph{Proc. Royal Soc. Edinburgh Sect A} \href{ https://doi.org/10.1017/prm.2018.75}.
	
\bibitem{KimLeeMartin-JMAA2015} \textsc{S.~K.~Kim, H.~J.~Lee, and M.~Mart\'in}, The Bishop-Phelps-Bollob\'{a}s theorem for operators from $\ell_1$ sums, \emph{J. Math. Anal. Appl.} \textbf{428} (2015), 920--929.

\bibitem{KLM} \textsc{S.~K.~Kim, H.~J.~Lee, and M.~Mart\'in}, On the Bishop-Phelps-Bollob\'{a}s theorem for operators and numerical radius,
 \emph{Studia Math.} {\bf 233} (2016), 141--151.

\bibitem{L} \textsc{J.~Lindenstrauss}, On operators which attain their norm, \emph{Israel J. Math.} {\bf 1} (1963), 139--148.

\bibitem{MPR1} \textsc{J.~F.~Mena, R.~Pay\'a and A.~Rodr\'iguez}, Semisummands and semiideals in Banach spaces, \emph{Israel J. Math.}, {\bf 52}, (1985), 33--67.

\bibitem{MPR2} \textsc{J.~F.~Mena, R.~Pay\'a and A.~Rodr\'iguez}, Absolute subspaces of Banach spaces, \emph{Quart. J. Math.}, {\bf 40}, (1989), 33--37.

\bibitem{MPR3} \textsc{J.~F.~Mena, R.~Pay\'a and A.~Rodr\'iguez}, Absolutely proximinal subspaces of Banach spaces, \emph{J. Approx. Theory} {\bf 65} (1991), 46--72.
	
\bibitem{Paya} \textsc{R.~Pay\'a}, {\it T\'ecnicas de Rango Num\'erico y Estructura en Espacios Normados}, Ph.D.\ Dissertation, Universidad de Granada, Spain, 1980. Available at http://hdl.handle.net/10481/52674. 

\bibitem{Paya1} \textsc{R.~Pay\'a}, A counterexample on numerical radius attaining operators, \emph{Israel J. Math.} {\bf 79} (1992), 83--101.


\bibitem{PS} \textsc{R.~Pay\'a and Y.~Saleh}, Norm attaining operators from $L_1(\mu)$ into $L_{\infty}(\nu)$, \emph{Arch. Math.} {\bf 75} (2000),  380--388.

\end{thebibliography}
\end{document}